\documentclass[11pt]{amsart}
\usepackage{graphicx}
\usepackage{amssymb}
\usepackage[usenames,dvipsnames]{color}
\numberwithin{equation}{section} 

\usepackage{mystyle}

\title{$C^{\s+\a}$ estimates for concave, non-local parabolic equations with critical drift}
\author[H. Chang Lara]{H\'ector Chang Lara}
\address{%
Columbia University\\
Department of Mathematics\\
New York, NY 10027
}
\email{changlara@math.columbia.edu}
\author[G. D\'avila]{Gonzalo D\'avila}
\address{%
University of British Columbia\\
Department of Mathematics\\
1984 Mathematics Road\\
Vancouver, B.C. Canada V6T 1Z2 
}
\email{gdavila@math.ubc.ca}

\begin{document}
\begin{abstract}
Given a concave integro-differential operator $I$, we study regularity for solutions of fully nonlinear, nonlocal, parabolic, concave equations of the form $u_t-Iu=0$. The kernels are assumed to be smooth but non necessarily symmetric which accounts for a critical non-local drift. We prove a $C^{\s+\a}$ estimate in the spatial variable and a $C^{1,\a/\s}$ estimates in time assuming time regularity for the boundary data. The estimates are uniform in the order of the operator $I$, hence allowing us to extend the classical Evans-Krylov result for concave parabolic equations.
\end{abstract}

\maketitle


\section{Introduction}
In this work we are interested in studying regularity of solutions of 
\begin{align}\label{eqintro}
u_t- \inf_{L \in \cL} Lu=0 \text{ in $B_1\times(-1,0]$},
\end{align}
where,
\begin{align*}
Lu(x) &:= (2-\s)\int \d u(x;y)\frac{K(y)}{|y|^{n+\s}}dy + b\cdot Du(x),\\
\d u(x;y) &:= u(x+y) - u(x) - Du(x)\cdot y\chi_{B_1}(y).
\end{align*}
The kernel $(2-\s)\frac{K(y)}{|y|^{n+\s}}$ is comparable to the fractional laplacian of order $\s \in [1,2)$ and it is non necessarily symmetric. As it was discussed in a previous paper \cite{CD3}, the odd part of the kernel brings drift terms after rescaling the equation. That is the reason why the drift term $b\cdot Du$ is included above. Contrasting to the second order case, where the lower order drift might be absorbed by estimates proven for pure second order equation at sufficiently small scales, in our case the drift remains comparable to the diffusion as the scales go to zero, making it critical.

This type of equations appear naturally when studying stochastic control problems (see \cite{Soner}), ergodic control problem (see \cite{MR}) and economic applications (see \cite{DP}), in which the the random part is given by a purely jump process, which is most of the time non necessarily symetric. The particular concave case can be seen as a one-player stochastic game, which at each step he can choose a strategy to minimize the expected value of some fixed function evaluated at the first exit point of a given domain. 
 
In the local case ($\s=2$) this problem was first studied independently by L. Evans and N. Krylov (see \cite{Evans}, \cite{Krylov} and also the recent proof by L. Caffarelli and L. Silvestre in \cite{C4}). They obtain $C^{2,\a}$ a priori estimates and therefore the existence of classical solutions by the continuity method.

L. Caffarelli and L. Silvestre proved in \cite{C3} that solutions of the elliptic problem $Iu=0$, where $I$ is a concave operator with smooth kernels, are $C^{\s+\a}$. It relies on the theory of viscosity solutions developed in \cite{C1} and \cite{C2}. The regularity obtained is enough to evaluate the operator in the classical sense. Moreover, the estimates obtained are independent of the order of the equation and extends the theory to the classical case.

A recent improvement of the previous work, done by J. Serra \cite{Serra14-2}, allowed to remove the smoothness condition for symmetric kernels in order to prove $C^{\s+\a}$ estimates. It proceeds by a compactness argument that blows-up the solution, reducing the problem to a Lioville type of result.

Regularity for parabolic nonlocal equations has been studied by the authors in \cite{CD}, \cite{CD2} and \cite{CD3} in which H\"older estimates are proven for general equations like \eqref{eqintro} with a non zero right hand side. Recent advances include the work of J. Serra \cite{Serra14} for $C^{1,\a}$ estimates with rough kernels; and the work of T. Jin, and J. Xiong, \cite{Jin14} for higher order, optimal Schauder estimates in the linear case.

We extend the ideas of \cite{C3} to the parabolic nonlocal case to prove the desired $C^{\s+\a}$ interior regularity. The order $\s$ is assumed at least one in order for the drift to be at most comparable with the diffusion. On the other hand, for $\s\in(0,1]$, the $C^{1,\a}$ estimates established in \cite{CD3} already give classical solutions. In contrast with the classical theory, where drift terms might be absorbed by estimates proven for pure second order equations at sufficiently small scales, our operator keeps the drift comparable to the diffusion as the scale go to zero providing new challenges.

We assume the boundary data to be at least differentiable in time. This is way to ensure that the solution is $C^{1,\a}$ in time in the interior. Keep in mind that for general boundary data one cannot expect $C^{1,\a}$ regularity in time, even for the fractional heat equation, an example is discussed in \cite{CD}. Whenever a weaker condition on the boundary data implies $C^{1,\a}$ regularity in time in the interior remains an open question.


Here is our main Theorem.

\begin{theorem}\label{thmintro}
Let $\s \in [1,2)$, $\cL \ss \cL_2^\s(\l,\L,\b)$ (sufficently smooth kernels to be defined) and $u$ satisfies in the viscosity sense,
\begin{align*}
u_t - \inf_{L \in \cL}Lu = 0 \text{ in $B_1\times(-1,0]$}.
\end{align*}
Then there is some $\a \in (0,1)$ and $C>0$, depending only on $n, \l, \L$ and $\b$ such that,
\begin{align*}
\|u\|_{C^{\s+\a}(B_{1/2}\times(-1,0])} \leq C(\|u\|_{L^1((-1,0]\mapsto L^1(\w_\s))}+[u\chi_{B_1^c}]_{C^{0,1}((-1,0]\mapsto L^1(\w_\s))})).
\end{align*}
\end{theorem}

The paper is divided as follows. In Section \ref{VSP} we introduce the family of operators we are considering, the notion of viscosity solution and recall some properties. We also state some previous results that we need for the rest of this work. 
We use the concavity of the non-linearity in Section \ref{EQLU} to determine an equation for the average of a given solution, in particular we get an equation for the fractional laplacian. In Section \ref{EstimateforLap} we use the previous equation to obtain a weak $C^\s$ estimate on the laplacian of the solution. Finally in Section \ref{Further} we prove a diminish of oscillation lemma for the fractional laplacian which implies our main theorem.


\section{Preliminaries and Viscosity Solutions}\label{VSP}

The cylinder of radius $r$, height $\t$ and center $(x,t)$ in $\R^n\times\R$ is denoted by $C_{r,\t}(x,t) := B_r(x)\times(t-\t,t]$. Whenever we omit the center we are assuming that they get centered at the origin in space and time.

Given the scaling properties of linear operators with non symmetric kernels discussed in \cite{Chang12} and \cite{CD3}, it is reasonable enlarge the family of linear operators to include (classical) drift terms. In this sense lets introduce the following notation where the time variable has been omitted as it is irrelevant for the computation:
\begin{align*}
L_{K,b}^\s u(x) &:= (2-\s)\int\d u(x;y)\frac{K(y)}{|y|^{n+\s}}dy + b\cdot Du(x),\\
\d u(x;y) &:=  u(x+y) - u(x) - Du(x)\cdot y\chi_{B_1}(y).
\end{align*}
Initially we may consider kernels bounded from above and away from zero:
\begin{align*}
0 < \l \leq K \leq \L < \8.
\end{align*}
The drift comes not only from the term $b\cdot D$ but also from the odd part of the kernel after rescaling. We assume that they are controlled in the following way:
\begin{align*}
\sup_{r \in (0,1)} \left|b+(2-\s)\int_{B_1\sm B_r} \frac{yK(y)}{|y|^{n+\s}}dy\right| \leq \b.
\end{align*}
We denote by $\cL_0^\s(\l,\L,\b)$ the family of all linear operators with the two conditions given above and suppress some its parameters in the notation to follow whenever they are clear from the context, usually we write just $\cL_0$. Sufficient regularity/integrability to evaluate $L_{K,b}^\s u(x)$ is $u \in C^{1,1}(x) \cap L^1(\w_\s)$ where $\w_\s(y) = \min(1,|y|^{-(n+\s)})$.

More regular kernels can be considered in order to prove higher regularity estimates. This corresponds to the initial approach taken in \cite{C1}, \cite{C2} and \cite{C3} in order to use integration by parts techniques to control rough oscillations of the boundary data of the solution. This work follows uses the same technique for which we define the family $\cL_1^\s(\l,\L,\b) \ss \cL_0$ such that for each kernel,
\begin{align*}
|DK(y)| \leq \L|y|^{-1}.
\end{align*}
Moreover, let $\cL_2^\s(\l,\L,\b) \ss \cL_1$ such that for each kernel,
\begin{align*}
|D^2K(y)| \leq \L|y|^{-2}.
\end{align*}
Lets remind that the smoothness hypothesis of the previous works have been lifted in \cite{Serra14} and \cite{Serra14-2} for symmetric kernels. Their techniques applies also if drifts or lower order terms are included because of scaling considerations. In the present case however, an odd kernel renews the drift a may keep it comparable to the diffusion even as the scales go to zero, so our result is not clearly contained in such work. 

Given $\cL\in\cL_0$, a non linearity $I$ is given by a function $I:\W\times(t_1,t_2]\times\R^\cL\to \R$ such that,
\begin{align*}
Iu(x,t) := I(x,t,(Lu(x,t))_{L\in\cL}).
\end{align*}
$I$ is considered to be elliptic if it is increasing in $\R^\cL$.

The nonlinearity in our main Theorem is constructed from $\cL \ss \cL_2$ such that,
\begin{align*}
Iu = \cM^-_\cL u := \inf_{L \in \cL} Lu.
\end{align*}
It satisfies the following uniform ellipticity relation with the extremal ope-rators,
\begin{align*}
\cM^-_\cL (u-v) &\leq Iu - Iv \leq \cM^+_\cL(u-v).
\end{align*}
where $\cM^+_{\cL}:= \sup_{L\in\cL}L$.

\subsection{Viscosity solutions}

We recall some definitions pertaining to viscosity solutions $u$ for the equation $u_t-Iu=f$. A test function $\varphi$ needs to be sufficiently smooth/integrable about the contact point where the equation is tested. Moreover, qualitative properties as the continuity of $Iu$, require for the tail of $u$ to be at least continuous in time in the following integrable sense.

\begin{definition}
The space $C((t_1,t_2] \mapsto L^1(\w_\s))$ consists of all measurable functions $u:\R^n\times(t_1,t_2] \to \R$ such that for every $t \in (t_1,t_2]$,
\begin{enumerate}
\item $\|u(\cdot,t)^-\|_{L^1(\w_\s)} < \8.$
\item $\lim_{\t\nearrow0}\|u(\cdot,t) - u(\cdot,t-\t)\|_{L^1(\w_\s)} = 0$.
\end{enumerate}
\end{definition}

\begin{definition}[Test functions]\label{def:test_function}
A test function is defined as a pair $(\varphi, C_{r,\t}(x,t))$, such that $\varphi \in C^{1,1}_xC^1_t(C_{r,\t}(x,t)) \cap C((t-\t,t]\mapsto L^1(\w_\s))$.
\end{definition}

Whenever the cylinder in the Definition \ref{def:test_function} becomes irrelevant we will refer to the test function $(\varphi, C_{r,\t}(x,t))$ just by $\varphi$.

\begin{definition}
Given a function $u$ and a test function $\varphi$, we say that $\varphi$ touches $u$ from below at $(x,t)$ if,
\begin{enumerate}
\item $\varphi(x,t)=u(x,t)$,
\item $\varphi(y,s) \leq u(y,s)$ for $(y,s)\in \R^n\times(t-\t,t]$.
\end{enumerate}
\end{definition}

A similar definition for contact from above will be considered too.

\begin{definition}[Viscosity (super) solutions]\label{viscosity}
Given an elliptic operator $I$ and a function $f$, a function $u \in C(\W\times(t_1,t_2]) \cap C((t_1,t_2]\mapsto L^1(\w_\s))$ is said to be a viscosity super solution to $u_t - Iu \geq f$ in $\W\times(t_1,t_2]$, if for every lower continuous test function $(\varphi,C_{r,\t}(x,t))$ touching $u$ from below at $(x,t) \in \W\times(t_1,t_2]$, we have that $\varphi_{t^-}(x,t) - I\varphi(x,t) \geq f(x,t)$.
\end{definition}

Recall that $\varphi_{t^-}$ denotes the left time derivative of $\varphi$ natural for time evolution problems.

The definition of $u$ being a viscosity \textit{sub} solution to $u_t - Iu \leq f$ in $\W\times(t_1,t_2]$ is done similarly to the definition of super solution replacing contact from below by contact from above and reversing the last inequality. A viscosity solution to $u_t - Iu = f$ in $\W\times(t_1,t_2]$ is a function which is a super and a sub solution simultaneously.


\subsection{Previous Results}

Several qualitative results for viscosity solutions of our parabolic equations such as the stability, a comparison principle and the existence of (viscosity) solutions have been stablished in \cite{CD}, \cite{CD2}, \cite{CD3}. We recall at this point some quantitative estimates for the solutions which will be used in this work.

\begin{theorem}[Point Estimate]\label{PE}
Let $\s \in [1,2)$. Suppose $u \geq 0$ satisfies 
\begin{align*}
u_t - \cM^-_{\cL_0}u &\geq -f(t) \text{ in $C_{2r,2r^\s}(0,r^\s)$}.
\end{align*}
Then, for every $s \geq 0$, 
\begin{align*}
\frac{|\{u > s\} \cap C_{r,r^\s}|}{|C_{r,r^\s}|} \leq C\1\inf_{C_{r,r^\s}(0,r^\s)}u + r^\s\fint_{-r^\s}^{r^\s}f^+(s)ds\2^{\e}s^{-\e},
\end{align*}
for some constants $\e$ and $C$ depending only on $n, \l, \L$ and $\b$.
\end{theorem}

The Oscillation Lemma provided in \cite{CD3} controls the point-wise size of a non negative sub solution in terms of an integral norm.


\begin{lemma}[Oscillation Lemma]\label{lem:oscillation}
Let $\cL\ss \cL_0$, $I:\W\times(t_1,t_2]\times\R^\cL\to\R$ uniformly elliptic and such that $I0=0$. Let $u$ satisfies,
\begin{align*}
u_t - I u &\leq f \text{ in $\W\times(t_1,t_2]$}.
\end{align*}
Then for every $\W'\times(t_1',t_2] \cc \W\times(t_1,t_2]$,
\begin{align*} 
\sup_{\W'\times(t_1',t_2]} u^+ \leq C\1\|u^+\|_{L^1((t_1,t_2] \mapsto L^1(\w_\s))} + \|f^+\|_{L^1((t_1,t_2]\mapsto L^\8(\W))}\2,
\end{align*}
for some universal $C>0$, independent of $\s \in [1,2)$, depending on the domains.
\end{lemma}


\begin{theorem}[H\"older regularity]
Let $u$ satisfies
\begin{alignat*}{2}
 u_t - \cM^+_{\cL_0}u &\leq f(t) &&\text{ in $C_{1,1}$},\\
 u_t - \cM^-_{\cL_0}u &\geq -f(t) &&\text{ in $C_{1,1}$},
\end{alignat*}
Then there is some $\a \in (0,1)$ and $C>0$, depending only on $n$, $\l$, $\L$ and $\b$, such that for every $(y,s), (x,t) \in C_{1/2,1/2}$
\begin{align*}
\frac{|u(y,s) - u(x,t)|}{(|x-y| + |t-s|^{1/\s})^\a} \leq C\1 \|u\|_{L^1((-1,0]\mapsto L^1(\w_\s))} + \|f\|_{L^1(0,1)}\2.
\end{align*}
\end{theorem}

\begin{theorem}[Regularity for translation invariant operators]\label{C1aold}
Let $\cL \ss \cL_1$, $I:\R^\cL\to\R$ be uniformly elliptic, translation invariant and such that $I0=0$. Let $u$ satisfies
\begin{alignat*}{2}
u_t - Iu = f(t) \text{ in $C_1$}.
\end{alignat*}
Then there is some $\a \in (0,1)$ and $C>0$, depending only on $n$, $\l$, $\L$ and $\b$, such that for every $(y,s), (x,t) \in C_{1/2,1/2}$,
\begin{align*}
|Du(x,t)|+\frac{|Du(x,t)-Du(y,s)|}{(|x-y| + |t-s|^{1/\s})^\a} \leq C\1\|u\|_{L^1((-1,0]\mapsto L^1(\w_\s))} + \|f\|_{L^1(0,1)}\2.
\end{align*}
\end{theorem}

The previous Theorem does not give more regularity in time even if $I$ is translation invariant in time and $f \equiv 0$. In \cite{CD} the authors gave an example of a function, not better than Lipschitz in its time variable, solving the fractional heat equation. However, further regularity in time can be retrieved via the Oscillation Lemma if the Dirichlet data has a smoothness condition controlled by,
\begin{align*}
[u]_{C^{0,1}((t_1,t_2]\mapsto L^1(\w_\s))} &:= \sup_{(t-\t,t] \ss (t_1,t_2]} \frac{\|u(t)-u(t-\t)\|_{L^1(\w_\s)}}{\t}.
\end{align*}


\begin{theorem}[Further regularity in time]\label{furthertime}
Let $\cL \ss \cL_0$, $I:\R^\cL\to\R$ be uniformly elliptic, translation invariant such that $I0=0$. Let $u$ satisfies
\begin{align*}
u_t-Iu = 0 \text{ in $C_{1,1}$}.
\end{align*}
Then there is some $\a \in (0,1)$ and $C>0$, depending only on $n$, $\l$, $\L$ and $\b$, such that for every $(x,t), (y,s)\in C_{1/2,1/2}$ we have 
\begin{align*}
|u_t(x,t)| + \frac{|u_t(x,t)-u_t(y,s)|}{(|x-y|+|t-s|^{1/\s})^\a}\leq C[u]_{C^{0,1}((-1,0]\mapsto L^1(\w_\s))}.
\end{align*}
\end{theorem}

\section{Equations for $Lu$ by concavity and translation invariance}\label{EQLU}

For this part we fix $\s\in[1,2)$, $\cL \ss \cL_2$, and $u$ such that,
\begin{align*}
&u_t - \cM^-_\cL u = 0 \text{ in $C_{8,3}$},\\
&\|u\|_{L^\8((-3,0] \mapsto L^1(\w_\s))} + [u]_{C^{0,1}((-3,0] \mapsto L^1(\w_\s))} \leq 1.
\end{align*}
We can assume that $u$ is a classical solution with smooth boundary and initial data. Otherwise, we approximate $u$ by a sequence of classical solutions with smooth boundary and initial data and recover the estimates of this chapter in the limit by the regularization procedure described in \cite{CD2}. The only thing we need to be careful about is that those a priori estimates are independent of (fractional) derivatives of $u$ which are not accessible by the viscosity solutions.

Many results in this and the following sections can be obtained by controlling $\|u\|_{L^1((-5,0] \mapsto L^1(\w_\s))}$, instead of the $L^\8$ norm; however, when it is coupled with the bound for $[u]_{C^{0,1}((-5,0] \mapsto L^1(\w_\s))}$ it actually implies $L^\8$ bound.

It is convenient for this section to introduce the following notation, given $K(y) \geq 0$ let
\begin{align*}
K^\s(y) := (2-\s)\frac{K(y)}{|y|^{n+\s}}.
\end{align*}
We denote the convolution by,
\begin{align*}
v \ast w(x) &:= \int w(x+y)v(y)dy.
\end{align*}
In particular, given that $K \geq 0$ goes to zero about the origin with at least a quadratic rate, then we can decompose a linear operator as:
\begin{align*}
L_{K,b}^\s u = \1K^\s \ast {} - \|K^\s\|_1 - \1\int_{B_1} yK^\s(y)dy - b\2\cdot D\2u.
\end{align*}

\begin{property}
Let $\a\in\R$, $b \in \R^n$ and $\eta \geq 0 \in L^1(\R^n)$. Then the following holds for any regular function $v$,
\begin{enumerate}
\item \textbf{Homogeneity:} $\cM^\pm_\cL(\a v) = \a\cM^\pm_\cL v$.
\item \textbf{Translation:} $\cM^-_\cL (b\cdot D v) \leq b\cdot D \cM^\pm_\cL v \leq \cM^+_\cL (b\cdot D v)$.
\item \textbf{Concavity:} $\eta\ast\cM^-_\cL v \leq \cM^\pm_\cL(\eta\ast v) \leq \eta\ast\cM^+_\cL v$.
\end{enumerate}
\end{property}

\begin{corollary}\label{cor:convolution}
For $K \geq 0$, $b \in \R^n$ and $\varphi\in C^\8_0(B_2\mapsto[0,1])$ such that $\varphi = 1$ in $B_1$, it holds that,
\begin{align*}
(L_{K,b}^\s u)_t - \cM^+_\cL(L_{K,b}^\s u) \leq ([(1-\varphi)K^\s]\ast u)_t - \cM^-_\cL ([(1-\varphi)K^\s]\ast u) \text{ in $C_{6,3}$}.
\end{align*}
In particular, if $\supp K \ss B_1$, then,
\begin{align*}
(L_{K,b}^\s u)_t - \cM^+_\cL(L_{K,b}^\s u) \leq 0 \text{ in $C_{6,3}$}.
\end{align*}
\end{corollary}

\begin{proof}
Let, for $\e \in (0,1)$, $K_\e := \chi_{B_\e^c}K$. We decompose the operator $L_{K_\e,b}^\s$ as a sum of a local and a non-local operator, the non-local being the one appearing on the right hand side of the conclusion of the lemma,
\begin{align*}
L_{K_\e,b}^\s &= L + NL,\\
&:= (L_{K_\e,b}^\s - K_\e^\s(1-\varphi)\ast) + K_\e^\s(1-\varphi)\ast,\\
&= \1\varphi K_\e^\s\ast{} - \|K_\e^\s\|_1 - \1\int_{B_1}yK_\e^\s(y)dy - b\2 \cdot D \2 + K_\e^\s(1-\varphi)\ast.
\end{align*}
Then,
\begin{align*}
(L_{K_\e,b}^\s u)_t - \cM^+_\cL(L_{K_\e,b}^\s u) &\leq \1(Lu)_t - \cM^+_\cL(Lu)\2 + \1(NLu)_t - \cM^-_\cL(NLu)\2,\\
&\leq L\1u_t - \cM^-_\cL u\2 + \1(NLu)_t - \cM^-_\cL(NLu)\2.
\end{align*}
In $C_{6,3}$ the first term is zero as the local operator $L$ does not take into account the values of $\1u_t - \cM^+_\cL u\2$ outside of $B_8$. The result now follows in the limit as $\e\searrow0$ by stability.
\end{proof}

\begin{property}[Integration by parts]
Let $K \geq 0$, $b \in \R^n$, $(\bar K(y), \bar b) := (K(-y), -b)$ and for $L = L_{K,b}^\s$, $\bar L = L_{\bar K,\bar b}^\s$. Then the following holds for any pair of regular/integrable functions $v$ and $w$,
\begin{align*}
\int vLw = \int w\bar L v.
\end{align*}
In particular,
\begin{align*}
L(v\ast w) = v\ast(Lw) = (\bar L v)\ast w.
\end{align*}
\end{property}

\begin{corollary}\label{cor:convolution2}
For $L_{K,b}^\s \in \cL_2$ it holds that,
\begin{align*}
(L_{K,b}^\s u)_t - \cM^+_\cL(L_{K,b}^\s u) &\leq C \text{ in $C_{6,3}$}.
\end{align*}
for some universal constant $C>0$. 
\end{corollary}

\begin{proof}
Corollary \ref{cor:convolution} tells us that it suffices to estimate $([(1-\varphi)K^\s]\ast u)_t - \cM^-_\cL ([(1-\varphi)K^\s]\ast u)$ in $C_{6,3}$,
\begin{align*}
([(1-\varphi)K^\s]\ast u)_t &= [(1-\varphi)K^\s]\ast u_t,\\
&\leq C[u]_{C^{0,1}((-3,0] \mapsto L^1(\w_\s))},\\
&=C,\\
\cM^-_\cL ([(1-\varphi)K^\s]\ast u) &\geq \inf_{L \in \cL_2}L([(1-\varphi)K^\s]\ast u),\\
&= \inf_{L \in \cL_2}(\bar L [(1-\varphi)K^\s]\ast u),\\
&\geq -C.
\end{align*}
In the last inequality we used that $|DK(y)| \leq \L|y|^{-1}$, $|D^2K(y)| \leq \L|y|^{-2}$ and $\|u\|_{L^\8((-3,0] \mapsto L^1(\w_\s))} \leq 1$.
\end{proof}

From now on we denote, for $r_1 > r_2 > 0$, $\psi_{r_1,r_2}\in C^\8_0(B_{r_1} \to [0,1])$ such that $\psi_{r_1,r_2} = 1$ in $B_{r_2}$.

\begin{corollary}\label{cor:convolution3}
Let $6 \geq r_1 > r_2 >0$, $K \geq 0$, $b\in\R^n$ such that either $L_{K,b}^\s \in \cL_2$ or $|b| \leq \b'$, $\supp K \ss B_1$ and $K(y) \in [0,\L']$, then,
\begin{align*}
(\psi_{r_1,r_2} L_{K,b}^\s u)_t-\cM^+_\cL(\psi_{r_1,r_2} L_{K,b}^\s u) &\leq C \text{ in $C_{r_2,3}$}.
\end{align*}
for some universal constant $C>0$ depending also on $r_1,r_2,\b'$ and $\L'$.
\end{corollary}

\begin{proof}
We use either Corollary \ref{cor:convolution} or \ref{cor:convolution2} to get that $\psi_{r_1,r_2} L_{K,b}^\s u$ satisfies the following inequality in $C_{r_2,3}$,
\begin{align*}
(\psi_{r_1,r_2} L_{K,b}^\s u)_t-\cM^+_\cL(\psi_{r_1,r_2} L_{K,b}^\s u) &\leq C + \sup_{L \in \cL_2} L((1-\psi_{r_1,r_2})L_{K,b}^\s u),\\
&= C + \sup_{L \in \cL_2} K_L \ast((1-\psi_{r_1,r_2})L_{K,b}^\s u)
\end{align*}
Where $K_L$ is the kernel associated to $L \in \cL_2$, notice the cancellations provided by the fact that $(1-\psi_{r_1,r_2})$ and its gradient are zero in $B_{r_2}$. Now we take a closer look at $[K_L \ast((1-\psi_{r_1,r_2})L_{K,b}^\s u)](x,t)$ for $(x,t) \in C_{r_2,3}$,
\begin{align*}
[K_L\ast((1-\psi_{r_1,r_2})L_{K,b}^\s u)](x,t) &= [(K_L(1-\psi_{r_1,r_2}(x+\cdot))) \ast L_{K,b}^\s u](x,t),\\
&= [L_{\bar K, \bar b}^\s (K_L(1-\psi_{r_1,r_2}(x+\cdot))) \ast u](x,t),\\
&\leq C.
\end{align*}
In the last inequality we used that $|DK_L(y)| \leq \L|y|^{-1}$, $|D^2K_L(y)| \leq \L|y|^{-2}$ and $\|u\|_{L^\8((-3,0] \mapsto L^1(\w_\s))} \leq 1$.
\end{proof}


\section{Estimate for $\D^{\s/2}u$}\label{EstimateforLap}

We keep the same assumptions as before in this part: $\s\in[1,2)$, $\cL \ss \cL_2$, and $u$ such that,
\begin{align*}
&u_t - \cM^-_\cL u = 0 \text{ in $C_{8,3}$},\\
&\|u\|_{L^\8((-3,0] \mapsto L^1(\w_\s))} + [u]_{C^{0,1}((-3,0] \mapsto L^1(\w_\s))} \leq 1.
\end{align*}

\begin{lemma}\label{lem:bound_laplacian}
For $K(y) \in [0,\L], b \in B_\b$,
\begin{align*}
\|L_{K,b}^\s u\|_{L^\8(C_{1,1})} \leq C,
\end{align*}
for some universal constant $C$.
\end{lemma}

\begin{proof}
We do it in several steps. Here is a summary of the strategy:
\begin{enumerate}
\item For $L \in \cL$, we bound $Lu$ from below by using the equation for $u$ and the control we have for $u_t$ inside the domain.
\item For $L \in \cL$, we integrate by parts to control $\|Lu(t)\|_{L^1(\w_\s)}$ and then apply Lemma \ref{lem:oscillation} to bound $Lu$ from above.
\item For general $K$ and $b$, we use $L^2$ theory to control $\|L_{K,b}^\s u(t)\|_{L^1(\w_\s)}$ and then apply Lemma \ref{lem:oscillation} to bound $L_{K,b}^\s u$ from above.
\item For general $K$ and $b$, we apply the previous step to 
\begin{align*}
(K'',b'') = \L(K',b') - \l(K,b),
\end{align*}
 with $L_{K',b'}^\s \in \cL$ to bound $L_{K,b}^\s u$ from below.
\end{enumerate}

\textbf{Step 1:} $L \in \cL$, then $Lu \geq -C$ in $C_{8,3}$.

It follows from the equation for $u$ and the regularity in time,
\begin{align*}
Lu \geq \cM^-_\cL u = u_t \geq -[u]_{C^{0,1}((-3,0] \mapsto L^1(\w_\s))}.
\end{align*}

\textbf{Step 2:} $L \in \cL$, then $Lu \leq C$ in $C_{4,2}$.

We apply the Oscillation Lemma to $\psi_{6,5} Lu$. By Corollary \ref{cor:convolution3}, $\psi_{6,5} Lu$ satisfies,
\begin{align*}
(\psi_{6,5} Lu)_t-\cM^+_\cL(\psi_{6,5} Lu) &\leq C \text{ in $C_{5,3}$}.
\end{align*}
We estimate now $\|(\psi_{6,5} Lu)^+\|_{L^1((-3,0]\mapsto L^1(\w_\s))}$. As $\psi_{6,5} Lu$ is bounded from below and compactly supported all we need is to control the following integral
\begin{align*}
\int_{-3}^0 \int \psi_{6,5} Lu &= \int_{-3}^0 \int \1\bar L\psi_{6,5}\2u \leq C.
\end{align*}
By Lemma \ref{lem:oscillation}, we have that $\psi_{6,5} Lu$ is bounded from above in $C_{4,2}$ where it coincides with $Lu$.

\textbf{Step 3:} Given $K(y) \in [0,\L']$ and $b \in B_{\b'}$ then $L_{K,b}^\s u \leq C$ in $C_{1,1}$.

Given $L \in \cL$ the previous steps tell us that $Lu$ is bounded in $C_{4,2}$, from Fourier analysis techniques we get then that (see Theorem 4.3 in \cite{C4}),
\begin{align*}
&\|L_{K,b}^\s u(t)\|_{L^2(B_2)} \leq C\|Lu(t)\|_{L^2(B_3)} \leq C,\\
\Rightarrow \quad &\|L_{K\chi_{B_1},b}^\s u(t)\|_{L^1(B_2)} \leq C + \|L_{K\chi_{B_1^c},0}^\s u(t)\|_{L^1(B_2)},\\
\Rightarrow \quad &\|\psi_{3,2}L_{K\chi_{B_1},b}^\s u\|_{L^1((-2,0]\mapsto L^1(\w_\s))} \leq C.
\end{align*}

By Corollary \ref{cor:convolution3}, $\psi_{3,2} L_{K\chi_{B_1},b}^\s u$ satisfies,
\begin{align*}
(\psi_{3,2} L_{K\chi_{B_1},b}^\s u)_t-\cM^+_\cL(\psi_{3,2} L_{K\chi_{B_1},b}^\s u) &\leq C \text{ in $C_{2,2}$},\\
\int_{-4}^0 f(t)dt &\leq C.
\end{align*}

By Lemma \ref{lem:oscillation}, $\psi_{3,2} L_{K\chi_{B_1},b}^\s u$ gets bounded from above in $C_{1,1}$. By the hypoteses we also obtain the bound for $\psi_{3,2} L_{K,b}^\s u$ in $C_{1,1}$ where it coincides with $L_{K,b}^\s u$,
\begin{align*}
\psi_{3,2} L_{K,b}^\s u \leq C + \psi_{3,2} L_{K\chi_{B_1^c},b}^\s u \leq C + \|u\|_{L^\8((-1,0] \mapsto L^1(\w_\s))}.
\end{align*}

\textbf{Step 4:}  Given $K(y) \in [0,\L]$ and $b \in B_{\b}$ then  $L_{K,b}^\s u \geq -C$ in $C_{1,1}$.

Consider $L_{K',b'}^\s \in \cL$ and $L_{K'',b''}^\s := \L L_{K',b'}^\s - \l L_{K,b}^\s$ such that $|b''| \leq (\L+\l)\b$,
and $K''(y) \in [0,\L^2]$. Given the result from the second step, it suffices to show that $L_{K'',b''}^\s u \geq -C$ in $C_{1,1}$. This now is just a consequence of applying the third step to $L_{K'',b''}^\s u$.
\end{proof}

\begin{corollary}
There is a universal constant $C>0$ such that,
\begin{align*}
(2-\s)\int\frac{|\d u(x,t;y)|}{|y|^{n+\s}}dy \leq C \text{ in $C_{1,1}$}.
\end{align*}
\end{corollary}

In particular, by Morrey estimates, we have that $u \in C^\a_x(C_{1,1})$ for every $\a \in [1,\s)$, see \cite{Stein}.

\begin{proof}
Using $K(y) := \L(2-\s)|y|^{-(n+\s)}$ in the previous Lemma we get,
\begin{align*}
(2-\s)\int\frac{\d u(x,t;y)}{|y|^{n+\s}}dy \geq -C \text{ in $C_{1,1}$}.
\end{align*}
Fixing $(x,t) \in C_{1,1}$ and using $K(y) := \L(2-\s)\sign(\d u(x,t;y))|y|^{-(n+\s)}$ in the previous Lemma we get,
\begin{align*}
(2-\s)\int\frac{\d^+u(x,t;y)}{|y|^{n+\s}}dy \leq C.
\end{align*}
Adding them up we conclude the Corollary.
\end{proof}

\section{Further regularity}\label{Further}

Regularity $C^{2,\a}$ can be reduced to H\"older regularity of the laplacian. The same holds with respect to $C^{\s+\a}$ regularity and $(-\D)^{\s/2}$, which suits well for non-local equations. On the other hand, $(-\D)^{\s/2}u$ can be thought as a difference of an average of $u$ with itself which relates with the concavity of $\cM^-_\cL$ in a proper way. We will exploit these two facts in this section to prove our $C^{\s+\a}$ regularity result.

We keep the previous hypothesis for this section, $\cL \ss \cL_2$ and $u$ satisfies,
\begin{align*}
&u_t - \cM^-_\cL u = 0 \text{ in $C_{8,3}$},\\
&\|u\|_{L^\8((-3,0] \mapsto L^1(\w_\s))} + [u]_{C^{0,1}((-3,0] \mapsto L^1(\w_\s))} \leq 1.
\end{align*}
In particular we know by now that, for $K(y) \in [0,\L], b \in B_\b$,
\begin{align*}
\|L_{K,b}^\s u\|_{L^\8(C_{1,1})} \leq C.
\end{align*}

Given $A\ss B_1$, let
\begin{align*}
K_A^\s(y) := (2-\s)\frac{\chi_A(y)}{|y|^{n+\s}},
\end{align*}
Fix $\varphi \in C_0^\8(B_1 \to [0,1])$ such that $\varphi = 1$ in $B_{1/2}$ and define
\begin{align*}
w_A(x) :=\varphi(x)\int (\d u(x;y) - \d u(0;y))K_A^\s(y)dy.
\end{align*}
By the properties deduced in the previous sections we have that $w_A$ is glo-bally bounded and satisfies in $C_{1,1}$,
\begin{align*}
(w_A)_t-\cM^+_\cL w_A\leq C.
\end{align*}

Lets consider also the extremal functions,
\begin{align*}
P(x) := \sup_{A\ss B_1} w_A &= (2-\s)\varphi(x)\int_{B_{1/2}}\frac{(\d u(x;y)-\d u(0;y))^+}{|y|^{n+\a}}dy,\\
N(x) := \sup_{A\ss B_1} (-w_A) &= (2-\s)\varphi(x)\int_{B_{1/2}}\frac{(\d u(x;y)-\d u(0;y))^-}{|y|^{n+\a}}dy.
\end{align*}

Our goal is to prove a diminish of oscillation lemma for $P+N$. This implies that $(-\D)^{\s/2}u$ is H\"older continuous and therefore the $C^{\s+\a}$ regula-rity. We start by proving that $P$ and $N$ are comparable modulus a controlled error.



\begin{lemma}\label{lema91}
There exist universal constants $C>0$ and $\a \in (0,1)$ such that for $(x,t)\in C_{1/8,1/2}$ we have,
\begin{align*}
\frac{\l}{\L}N - C|x|^\a\leq P \leq \frac{\L}{\l}N + C|x|^\a.
\end{align*}
\end{lemma}

\begin{proof}
For $x\in B_{1/8}$, let $u_x(y) = u(x+y)$. Since $u$ solves $u_t-\cM^-_\cL u=0$ in $C_{1,1}$, then difference $(u_x-u)$ satisfies in $C_{7/8,1}$,
\begin{align*}
(u_x-u)_t-\cM^+_\cL(u_x-u)&\leq 0,\\
(u_x-u)_t-\cM^-_\cL(u_x-u)&\geq 0.
\end{align*}

To recover $P$ and $N$ from the previous relations we consider for $L=L_{K,0}^\s\in\cL_2$,
\begin{align*}
L(u_x-u)(0) &= \int(\d u(x;y)-\d u(0;y))K^\s(y)dy,\\
\l P(x) - \L N(x) &\leq \int_{B_1}(\d u(x;y)-\d u(0;y))K^\s(y)dy \leq \L P(x) - \l N(x).
\end{align*}
By change of variables,
\begin{align*}
\int_{B^c_1} = &\int u(y)\1K^\s(y-x)\chi_{B^c_1}(y-x) - K^\s(y)\chi_{B^c_1}(y)\2dy\\
&{} + (u(x)-u(0))\int_{B^c_1} K^\s(y)dy.
\end{align*}
By Theorem \ref{C1aold}, the last term is of order $|x|$. The first term can be estimated using the smoothness hypothesis of $K$,
\begin{align*}
&\int |K^\s(y-x)\chi_{B^c_1}(y-x) - K^\s(y)\chi_{B^c_1}(y)|dy,\\
\leq &\int_{B^c_{1/2}}|K^\s(y-x) - K^\s(y)|dz \leq C|x|.
\end{align*}
On the other hand we have the estimate $\|(u_x-u)_t\|_\8\leq C|x|^\a$ from Theorem \ref{furthertime}. Therefore,
\begin{align*}
[(u_x-u)_t-L(u_x-u)](0,t) &\geq -C|x|^\a - \L P(x) + \l N(x).
\end{align*}
Taking the infimum over $L\in\cL_2$ and using the equation for $(u_x-u)$ we get
\begin{align*}
0 &\geq -C|x|^\a - \L P(x) + \l N(x)
\end{align*}

A similar computation with $(u_x-u)_t-\cM^-_\cL(u_x-u)\geq 0$ provides the other inequality.
\end{proof}

The next result is a diminish of oscillation lemma. As we have learned from \cite{C1,C2,C3,CD,CD2} it is important to strengthen the hypothesis of being just bounded and allow some growth at infinity. This allows to iterate the lemma taking into account that the tails grow in a controlled way.

By rescaling we can further assume that for $\e_1>0$ sufficiently small (to be fixed) and for every set $K \ss \R^n$
\begin{alignat}{2}
\label{resc0} |w_K| &\leq 1/2 &&\text{ in $C_{1,1}$},\\
\label{resc1} |w_K| &\leq |x|^{1/2} &&\text{ in $B_1^c\times[-1,0]$},\\
\label{resc3} (w_K)_t-\cM^+_\cL w_K &\leq \e_1 &&\text{ in $C_{1,1}$}.
\end{alignat}
Additionally, by the previous lemma, we assume that in $C_{1/2,1}$,
\begin{align}
\label{resc4} \frac{\l}{\L}N(x,y)-\e_1|x|^\a &\leq P(x,t)\leq \frac{\L}{\l}N(x,t)+\e_1|x|^\a.
\end{align}


\begin{lemma}\label{Ca for P}
Assume \eqref{resc0}, \eqref{resc1}, \eqref{resc3} and \eqref{resc4}. There are constants $\kappa,\theta > 0$, sufficiently small, such that in $C_{\kappa,\kappa^\s}$
\begin{align*}
P \leq \frac{1}{2}-\theta.
\end{align*}
\end{lemma}

\begin{remark}\label{rescaling_of_osc_lemma}
We should ask ourselves how small should $\kappa$ and $\theta$ in order to be able to iterate the lemma. We need the rescaled $\tilde w_K$, given by
\begin{align*}
\tilde w_K(x,t) = \frac{w_K(\kappa x, \kappa^\s t)}{1-\theta},
\end{align*}
to satisfy the same hypothesis \eqref{resc0}, \eqref{resc1}, \eqref{resc3} and \eqref{resc4}. \eqref{resc0} is immediate, \eqref{resc1} holds if $(1-\theta) - \kappa^{1/2} \geq \theta/2 > 0$ which is reasonable as $\kappa,\theta$ can be chosen even smaller. \eqref{resc3} holds if $(1-\theta) > \kappa^\s$ which was already contained in the previous inequality as $\s > 1/2$. \eqref{resc4} holds if $\kappa^{\s-\a} \leq (1-\theta)$ which is possible because $\s > 1 > \a$.
\end{remark}


\begin{proof}
Assume by contradiction that for some $(x_0,t_0) \in C_{\kappa,\kappa^\s}$, $P(x_0,t_0) > (1/2-\theta)$. There is then some set $A$ such that $w_A(x_0,t_0) > (1/2-\theta)$. The function $v_A$, given by the following truncation,
\begin{align*}
v_A := \1\frac{1}{2} - w_A\2^+,
\end{align*}
satisfies an equation coming from \eqref{resc3}. As usual the truncation introduces an error that can be controlled in the interior
\begin{align*}
(v_A)_t - \cM_{\cL}^-v_A \geq -C \text{ in $C_{1/2,1}$}.
\end{align*}

We use the Point Estimate \ref{PE} to control the distribution of $v_A$ in $C_{\kappa,\kappa^\s}(0,-\kappa^\s)$,
\begin{align}\label{G}
\frac{|\{v_A > s\theta\} \cap C_{\kappa,\kappa^\s}(0,-\kappa^\s)|}{|C_{\kappa,\kappa^\s}(0,-\kappa^\s)|} \leq C(\theta + \kappa^\s)^\e(s\theta)^{-\e}.
\end{align}
By choosing $\kappa^\s\leq\theta$ we will make the right hand side $Cs^{-\e}$ sufficiently small, independently of $\theta$, by taking $s$ sufficiently large. This makes
\begin{align*}
G := \{w_A \geq (1/2 - s\theta)\} \cap C_{\kappa,\kappa^\s}(0,-\kappa^\s)
\end{align*}
to cover a fraction of $C_{\kappa,\kappa^\s}(0,-\kappa^\s)$ close to one.

In $G$, $w_A$ and $P$ are close to $1/2$. By \eqref{resc4}, $N$ can be forced also to be strictly positive in $G$, say larger than $\l/(4\L)$ by making $\e_1 + \theta \leq 1/4$. Also in $G$ and for $B = B_1\sm A$, $w_B$ has to be close to $-N$. This is because $w_A+w_B=P-N$, then 
\begin{align*}
0 \leq N+w_B=P-w_A \leq s\theta.
\end{align*}
This allows us to make $w_B \leq -\l/(8\L)$ in $G$ by choosing $\theta < \l/(8s\L)$.

Now we use the Oscillation Lemma to obtain the contradiction. Consider, for $\eta \in (0,1)$, $v_B$ given by,
\begin{align*}
v_B(x,t) = \1w_B(\kappa\eta x, (\kappa\eta)^\s t-\kappa^\s)+\frac{\l}{8\L}\2^+.
\end{align*}
It still satisfies in $C_{(\kappa\eta)^{-1},(\kappa\eta)^{-\s}}$, 
\begin{align*}
(v_B)_t-\cM^+_\cL(v_B) \leq \e_1(\kappa\eta)^\s \leq \e_1.
\end{align*}
Also, from \eqref{G}, we know that,
\begin{align*}
|\{v_B > 0\} \cap C_{\eta^{-1},\eta^{-\s}}| \leq C\eta^{-(n+\s)}s^{-\e}.
\end{align*}
By the Oscillation Lemma,	
\begin{align*}
\frac{\l}{8\L} = v_B(0,0) &\leq C\1 \e_1 + \eta^{-(n+\s)}s^{-\e} + \sup_{t\in[-\eta^{-\s},0]}\int_{B^c_{\eta^{-1}}} \frac{|v_B(y,t)|}{|y|^{n+\s}}dy\2.
\end{align*}
Changing variables,
\begin{align*}
\int_{B^c_{\eta^{-1}}} \frac{|v_B(y,t)|}{|y|^{n+\s}}dy &= (\kappa\eta)^\s\int_{B^c_\kappa} \frac{\1w_B(y, (\kappa\eta)^\s t-\kappa^\s)+\frac{\l}{8\L}\2^+}{|y|^{n+\s}}dy,\\
&\leq C\eta^\s,
\end{align*}
where the last inequality holds by the bounds \eqref{resc0} and \eqref{resc1}. Putting it back in the estimate we obtain,
\begin{align*}
\frac{\l}{8\L} \leq C\1\e_1 + \eta^{-(n+\s)}s^{-\e} + \eta^\s\2.
\end{align*}
This gives us a contradiction by choosing $\e_1,\eta^\s < \l/(100C\L)$ and then $s^\e > (100C\L)/(\l\eta^{n+\s})$.
\end{proof}

We are now able to prove the parabolic nonlocal Evans-Krylov Theorem.


\begin{theorem}[Classical solutions]
Let $\cL \ss \cL_2$, $u$ be a bounded function in $\R^n\times(-1,0]$ solving
\begin{align*}
u_t-\cM^-_\cL u=0 \text{ in viscosity in $C_{1,1}$},
\end{align*}
Then $(-\D)^\s u$ is H\"older continuous with the following estimate
\begin{align*}
\|(-\D)^\s u\|_{C^\a(C_{1/2,1/2})} \leq C(\|u\|_{L^\8((-1,0] \mapsto L^1(\w_\s))} + [u]_{C^{0,1}((-1,0] \mapsto L^1(\w_\s))}).
\end{align*}
\end{theorem}

\begin{proof}
The case $\s\leq 1$ is contained in \cite{CD3}. By the regularization procedure of \cite{CD2} we can assume that $(-\D)^\s u$ is continuous, all we need to show is the estimate at the origin. As usual we re-normalize $u$ in order to have $\|u\|_{L^\8((-1,0] \mapsto L^1(\w_\s))} + [u]_{C^{0,1}((-1,0] \mapsto L^1(\w_\s))} \leq 1$

By the definitions of $P$ and $N$ we have the following identity in $B_{1/8}\times(-1,0]$
\begin{align*}
&(-\D)^\s u(0)-(-\D)^\s u(x) =\\
&C\1 P(x)+N(x)+(2-\s)\int_{B^c_1}\frac{\d u(x;y)-\d u(0;y)}{|y|^{n+\s}}dy\2.
\end{align*}
The third term can be bounded by $C|x|$ as in the proof of Lemma \ref{lema91}.

Lemma \ref{Ca for P} and the Remark \ref{rescaling_of_osc_lemma} gives a geometric decay for $P$ around the origin which implies a H\"older modulus of continuity for it. By Lemma \ref{lema91} this is equivalent to a similar modulus of continuity for $N$. Then, the first two terms above can be bounded by $C|x|^{\a}$, for some universal $\a$, which concludes the proof.
\end{proof}

\end{document}